\documentclass[12pt, reqno]{amsart}

\usepackage{amssymb,amscd,amsthm, verbatim,amsmath,color,fancyhdr, mathrsfs}
\usepackage{graphicx}
\usepackage{turnstile,cite}
\usepackage[plainpages=false,pdfpagelabels]{hyperref}
\usepackage{hyperref}
\usepackage{orcidlink}
\usepackage{setspace}
\usepackage{enumitem}
\hypersetup{colorlinks,%
citecolor=red,%
filecolor=black,%
linkcolor=blue,%
urlcolor=blue
}
\usepackage[varg]{pxfonts}

\allowdisplaybreaks
\usepackage[letterpaper, left=2.5cm, right=2.5cm, top=2.5cm,
bottom=2.5cm,dvips]{geometry}
\newtheorem{theorem}{Theorem}[section]
\newtheorem{corollary}[theorem]{Corollary}
\newtheorem{lemma}[theorem]{Lemma}

\theoremstyle{definition}
\newtheorem{definition}[theorem]{Definition}
\newtheorem{remark}[theorem]{Remark}
\newtheorem{example}[theorem]{Example}
\numberwithin{equation}{section}
\begin{document}
\title{Fixed Point Theorems for Hardy--Rogers Type Multivalued Mappings  in $b$-Metric Spaces}

%%
%% Now edit the following to give your name and address:
%%
\author{Hezekiah Seun Adewinbi ${^{*1}}$\orcidlink{0009-0003-0643-2146}, Oluwatosin Temitope Mewomo${^2}$\orcidlink{0000-0003-0389-7469}}
\address{${^1}$Department of Mathematical Sciences, Kent State University,
	Kent, OH, 44242, United States.}
\email{hadewinb@kent.edu}
\address{${^2}$Department of Mathematics,
	Tarleton State University,
	Stephenville, Texas,
	United States}
\email{omewomo@tarleton.edu}

\keywords{$b$-metric space, multivalued mappings, singlevalued mappings,  fixed point, strict fixed point, $p$-th iterates, $h$-upper semicontinuous, Hardy-Rogers type mappings  \\
		{\rm 2020} {\it Mathematics Subject Classification}: 54H25, 47H10\\
*Corresponding author: Hezekiah Seun Adewinbi (hadewinb@kent.edu)}

\begin{abstract}
	In this paper, we establish new fixed point theorems for multivalued mappings satisfying Hardy–Rogers type contractive conditions in complete $b$-metric spaces. Motivated by the constant-coefficient framework of Chaib et al. (Fixed Point Theory 27 (2026), 209–228), we develop a distance-dependent formulation in which the contractive coefficients are nonnegative functions of the interpoint distance. We introduce a piecewise contraction factor corresponding to the two cases determined by the relative sizes of the cross terms in the Hardy–Rogers inequality, thereby providing a formulation adapted to the structure of the contractive condition. Under appropriate assumptions on the coefficient functions, we establish fixed point existence through either $h$-upper semicontinuity or asymptotic limsup conditions. The resulting framework yields corresponding fixed point theorems for multivalued Reich, Kannan, and Chatterjea-type contractions. We further investigate strict fixed points and obtain conditions under which the fixed point and strict fixed point sets coincide and consist of a unique point. In addition, we establish a fixed point transfer result for suitable finite iterates, showing that the contractive condition may be imposed on $T^p$ rather than on the original mapping, provided that the latter has a strict fixed point. Corresponding single-valued consequences are also derived. Several examples illustrate the applicability of the results and demonstrate situations in which the proposed conditions apply beyond the corresponding existing formulations.
\end{abstract}

	\maketitle
	\section{Introduction}
	
The classical Banach contraction principle \cite{Banach} provides existence and uniqueness results for single-valued mappings in complete metric spaces. Over time, this theory has been extended to both single-valued and multivalued mappings to accommodate more complex mathematical models. Notable contributions include the work of Hardy and Rogers \cite{Hardy}, who generalized Kannan \cite{Kannan}, Chatterjea\cite{Chatterjea} and Reich \cite{Reich} contractive conditions by incorporating multiple distance terms. 

The Banach contraction principle was extended to the multivalued setting by Nadler \cite{Nadler} in 1969, who introduced multivalued contraction mappings from a complete metric space $X$ into $\mathcal{CB}(X)$, the family of all nonempty closed and bounded subsets of $X$. Subsequently, Czerwik et al.\cite{Czerwik1998} extended Nadler’s results to the more general framework of $b$-metric spaces.

In 1989, Bakhtin \cite{Bakhtin} and Czerwik \cite{Czerwik} introduced the concept of a $b$-metric space by relaxing the triangle inequality of standard metric spaces.  Since then, numerous authors have investigated fixed point theorems for both single-valued and multivalued mappings in metric and $b$-metric spaces (see \cite{Lael,  Derouiche,  Khojasteh, Leyew, Joonaghany, Aliouche, Karapinar2022, AliHussain, AliKhan, Errai, Patel, Aleksic, Chifu, Al-Izeri, Andres, Roldan2018
 }  and the references therein).

In a recent work, Chaib et al. \cite{Chaib2026} established several fixed point results for multivalued mappings satisfying a Hardy–Rogers type contraction in the framework of $b$-metric spaces, which improve upon the corresponding results of Chifu and Petruşel \cite{Chifu2017}.

Motivated by these developments, we establish new fixed point theorems for Hardy–Rogers type multivalued mappings in complete $b$-metric spaces. Our approach extends the constant-coefficient setting by allowing the contractive coefficients to depend on the interpoint distance. We introduce a piecewise contraction factor, together with explicit denominator conditions, that reflects the two cases arising naturally from the Hardy–Rogers inequality and provides a more precise formulation of the contraction mechanism considered by Chaib et al. \cite{Chaib2026}. We further establish fixed point results for suitable finite iterates and derive conditions under which the fixed point and strict fixed point sets coincide and consist of a unique point. In addition, $h$-upper semicontinuity is incorporated as an alternative to the asymptotic limsup conditions for fixed point existence. Several examples are presented to illustrate the applicability of the results and to exhibit settings beyond the corresponding existing formulations.
 
	\section{Preliminaries}
For the convenience of the reader, we collect here some basic facts and establish the notation and terminology that will be used in subsequent sections.
	
	\begin{definition}[Bakhtin \cite{Bakhtin} and  Czerwik \cite{Czerwik}]
		\label{b-metric space}
		Let $X$ be a nonempty set,  and let $s\geq 1$ be a given real number. A function $d:X\times X\to\mathbb{R}_+$ is called a $b$-metric if for all $x,y,z\in X$,
		\begin{enumerate}[label=(\alph*)]
			\item $d(x,y)=0$ if and only if $x=y$;
			\item $d(x,y)=d(y,x)$;
			\item \label{inequality}$d(x,z)\leq s[d(x,y)+d(y,z)]$.
		\end{enumerate}
		A pair $(X,d)$ is called a $b$-metric space. 
	\end{definition}

\begin{example}[Snowflake]
	Let $(X, d_0)$ be a metric space, and let $n > 1$. Define  
	$d(x, y) = d_0(x, y)^n.$ Thus, $(X,d)$ is a $b$-metric space with constant $s = 2^{n- 1}$.
\end{example}
 Let $(X,d)$ be a $b$-metric space with $s \geq 1$.  The following notions and properties can be found in \cite{Nadler,Covitz, Singh2005}. 
\begin{itemize}
	\item $P(X)$ is the family of all nonempty subsets of $X$.
	\item $\mathcal{C}(X)$ is the family of nonempty closed subsets of $X$.
	\item $\mathcal{CB}(X)$ is the family of nonempty bounded and closed subsets of $X$.
	\item $D(A, B) = \inf\{d(a, b) : a \in A, b \in B\}$, for all nonempty subsets $A, B$ of $X$. 
	\item $D (\{a\} , B) = D (a, B)$ whenever $A = \{a\}$.
	\item $\delta(A, B) =\displaystyle  \sup_{x \in A} D(x, B)$.
	\item $d(a, B) = \inf \{d(a, b): b\in B\}$.
\end{itemize}

\begin{definition}
	Let $X$ and $Y$ be nonempty sets. A mapping $T:X\to Y$ is \textit{multivalued }if $T$ is a function from $X$ to the power set of $Y$. i.e. $T:X\to 2^Y. $
\end{definition}
\begin{remark}
	Every single valued mapping can be viewed as a multivalued mapping. Indeed, let $f:X\to Y$ be a single valued mapping. Define $T:X\to 2^Y$ by $Tx=\{f(x)\}.$
\end{remark}
\begin{definition}
	Let $T : X \to \mathcal{C}(X)$ be a multivalued mapping.  A point $x \in X$ is called a \textit{fixed point} for $T$ if and only if $x \in Tx$. The set  $\text{Fix}(T) = \{x \in X \mid x \in Tx\}$ is called the fixed point set of $T$. 
\end{definition}

\begin{definition}
	Let $(X,d)$ be a $b$-metric space. The  \textit{generalized Hausdorff $b$-metric} $H$ on $\mathcal{C}(X)$, the collection of all non-empty closed subsets of $(X, d)$, is defined as follows:
	$$
	H(A, B) = 
	\begin{cases}
		\max\{\delta(A, B), \delta(B, A)\}, & \text{if the maximum exists;} \\
		\infty, & \text{otherwise},
	\end{cases}
	$$
	where $A, B \in  \mathcal{C}(X)$.
	If $A, B \in \mathcal{CB}(X)$, then $H(A, B) = \max\{\delta(A, B), \delta(B, A)\}$.
\end{definition}
\begin{remark}
	In this paper, we work with the class $\mathcal{C}(X)$ of nonempty closed subsets of $X$ rather than $\mathcal{CB}(X)$, since under the contractive assumptions considered the relevant Hausdorff distances are finite.
\end{remark}
\begin{definition}
	A multivalued mapping $ T : X \to \mathcal{C}(X) $ is \textit{continuous} if and only if, for each $ x \in X $ and $ \{x_n\} \subset X $ with $ \displaystyle\lim_{n \to \infty} d(x_n, x) = 0 $, we have $ \displaystyle \lim_{n \to \infty} H(Tx_n, Tx) = 0 $.
\end{definition}

\begin{definition}
	\label{Continuous}
	A multivalued mapping $ T : X \to \mathcal{C}(X) $ is \textit{$h$-upper semicontinuous} if and only if, for each $ x \in X $ and $ \{x_n\} \subset X $ with $ \displaystyle \lim_{n \to \infty} d(x_n, x) = 0 $, we have  
	$ \displaystyle 	\lim_{n \to \infty} \delta(Tx_n, Tx) = 0.$
\end{definition}

\begin{remark}
	\label{uppersemicontinuous}
	$T$ is continuous implies that it is $h$-upper semicontinuous.
\end{remark}

\begin{lemma}
	\label{lemmaalpha}
	Let $A$ and $B$ be distinct elements of $\mathcal{C}(X)$. Let $x \in A$. Then, for a given number $\eta > 1$, there exists a point $y \in B$ such that
	$
	d(x, y) \leq \eta H(A, B).
	$
\end{lemma}

\begin{lemma}
	\label{properties}
	Let $(X,d)$ be a $b$-metric space. For any, $A,B,C\in  \mathcal{C}(X)$ and any $x,y\in X$, one has the following:
	\begin{enumerate}[label=(\alph*)]
		\item $d(x,B)\leq d(x,b)$, for any $b\in B$;
		\item \label{14c} $D(x,B)\leq \delta(A,B)\leq H(A,B),$ for any $x\in A$;
		\item $H(A,C)\leq s(H(A,B)+H(B,C))$;
		\item $D(x,A)\leq s(d(x,y)+D(y,A))$.
		%\item $D(A, B) \leq H(A, B).$
	\end{enumerate}
\end{lemma}	

\begin{lemma}[Miculescu et al. \cite{Miculescu}]
	\label{lem}
	Let $(X,d)$ be a $b$-metric space with $s\geq 1$ and $\{x_n\}$ a sequence in $X$ such that 
	\begin{equation*}
		d(x_{n},x_{n+1})\leq rd(x_{n-1},x_{n}), ~~n=1,2,...
	\end{equation*}
	where $0\leq r<1$. Then $\{x_n\}$ is a Cauchy sequence. 
\end{lemma}

\section{Hardy--Rogers Type Multivalued  Mappings}
In this section, we establish fixed point theorems for Hardy–Rogers type multivalued mappings in complete $b$-metric spaces. We provide several examples to demonstrate the applicability of the obtained results and to highlight their relationship with existing fixed point theorems. As consequences of the main results, we derive corresponding fixed point theorems for Kannan-, Reich-, and Chatterjea-type multivalued mappings.
	\begin{definition} 
	Let $(X,d)$ be a $b$-metric space. A mapping $T : X \to \mathcal{C}(X)$ is said to be \textit{Hardy-Rogers type multivalued mapping} if there exist nonnegative functions $a,b,c,e,f:[0,\infty)\to [0,1)$ satisfying $a(t)+b(t)+c(t)+2s\min\{e(t),f(t)\}<1$ for all $t\geq 0$ such that,  for all $x,y\in X$
\begin{equation}
		\label{A3}
		\begin{aligned}
		H(Tx,Ty)&\leq a(d(x,y)) d(x,y)+b(d(x,y)) D(x,Tx)+ c(d(x,y))D(y,Ty)\\&\hspace{1.5cm}+e(d(x,y))D(x,Ty)+f(d(x,y))D(y,Tx)
	\end{aligned}
\end{equation}

\end{definition}

	\begin{theorem}
		\label{Main2}
	Let $(X,d)$ be a complete $b$-metric space and  let $T: X \to \mathcal{C}(X)$ be a Hardy-Rogers type multivalued mapping. Define 	the piecewise contraction factor by
\begin{eqnarray}
		\label{gamma}
	\gamma(t)=\begin{cases}	\displaystyle	\frac{a(t)+b(t)+se(t)}	{1-c(t)-se(t)},	& e(t)\leq f(t),\\[2ex]	\displaystyle	\frac{a(t)+c(t)+sf(t)}{1-b(t)-sf(t)},	& f(t)<e(t).\end{cases}
\end{eqnarray}
 Suppose that $\displaystyle \sup_{t \geq 0} \gamma(t)<1$.  Then $T$ has a fixed point in $X$ if at least one of the following conditions is satisfied:
		\begin{enumerate}[label=(\alph*)]
			\item $T$ is $h$-upper semicontinuous;
		\item $\displaystyle s	\limsup_{t\to0^+}\left(c(t)+e(t)\right)<1$;
	\item $\displaystyle s	\limsup_{t\to0^+}\left(b(t)+f(t)\right)<1$. 
		\end{enumerate}
	\end{theorem}
	\begin{proof}
	Choose $x_0\in X$ and $x_1\in Tx_0$. If $x_1\in Tx_1$, then
	$x_1$ is a fixed point of $T$ and the proof is complete. Hence, assume
	that $x_1\notin Tx_1$. We construct a sequence $\{x_n\}$ inductively
	such that $x_{n+1}\in Tx_n,$ for all $ n\geq1.$
	By Lemma~\ref{lemmaalpha}, for every $\eta>1$ and $n\geq1$, we can
	choose $x_{n+1}\in Tx_n$ such that
	\begin{equation}
		\label{A}
		d_{n+1}=d(x_n,x_{n+1})
		\leq
		\eta H(Tx_{n-1},Tx_n).
	\end{equation}
	Put $d_n=d(x_{n-1},x_n),~ n\geq1.$ For each $n\geq1$, either $e(d_n)\leq f(d_n)$ or $f(d_n)<e(d_n)$.	
	\paragraph*{\textit{Case 1}}
	Suppose that
	$e(d_n)\leq f(d_n).$
	Putting $x=x_{n-1}$ and $y=x_n$ in \eqref{A3} and then applying Lemma~\ref{properties}\ref{14c}, we have 
	\begin{align*}
		H(Tx_{n-1},Tx_n)
		&\leq
		a(d_n)d_n
		+b(d_n)d_n
		+c(d_n)D(x_n,Tx_n)
		+e(d_n)D(x_{n-1},Tx_n)
		+f(d_n)D(x_n,Tx_{n-1})\\
		&\leq
		(a(d_n)+b(d_n))d_n
		+c(d_n)H(Tx_{n-1},Tx_n)
		+se(d_n)
		\left(d_n+H(Tx_{n-1},Tx_n)\right).
	\end{align*}
	Consequently, $$
	\left(1-c(d_n)-se(d_n)\right)
	H(Tx_{n-1},Tx_n)
	\leq
	\left(a(d_n)+b(d_n)+se(d_n)\right)d_n.$$
	Moreover, since $a(d_n)+b(d_n)+c(d_n)+2se(d_n)<1,$
	we have $1-c(d_n)-se(d_n)>0.$
	Therefore,
	\begin{equation}
		\label{Aa}
		H(Tx_{n-1},Tx_n)	\leq\frac{a(d_n)+b(d_n)+se(d_n)}
		{1-c(d_n)-se(d_n)}\,d_n
		=\gamma(d_n)d_n.
	\end{equation}
	\paragraph*{\textit{Case 2}} On the other hand, suppose that $f(d_n)<e(d_n).$
Putting $x=x_{n}$ and $y=x_{n-1}$ in \eqref{A3}, and applying the
	$b$-metric inequality together with Lemma~\ref{properties}\ref{14c},
	we obtain
	\begin{equation}
		\label{Bb}
		H(Tx_n,Tx_{n-1})
		\leq
		\frac{a(d_n)+c(d_n)+sf(d_n)}
		{1-b(d_n)-sf(d_n)}\,d_n.
	\end{equation}
	Since the Hausdorff distance is symmetric and $
	a(d_n)+b(d_n)+c(d_n)+2sf(d_n)<1$, we get 
	\begin{equation}
		\label{Bb2}
		H(Tx_{n-1},Tx_n)
		\leq
		\frac{a(d_n)+c(d_n)+sf(d_n)}
		{1-b(d_n)-sf(d_n)}\,d_n
		=
		\gamma(d_n)d_n.
	\end{equation}
	Hence, in both cases, 
	\begin{equation}
		\label{B}
		H(Tx_{n-1},Tx_n)
		\leq
		\gamma(d_n)d_n, \text{ for all } n\geq1.
	\end{equation}
	Set $q=\displaystyle \sup_{t\geq0}\gamma(t)<1.$ Choose $\eta>1$ sufficiently close to $1$ such that $\beta=\eta q<1.$
	It follows from \eqref{A} and \eqref{B} that
	\begin{align*}
		d_{n+1}\leq
		\eta H(Tx_{n-1},Tx_n) \leq
		\eta\gamma(d_n)d_n\leq	\eta q\,d_n=
		\beta d_n.
	\end{align*}
	Therefore, by Lemma~\ref{lem}, the sequence $\{x_n\}$ is Cauchy.
	Since $(X,d)$ is complete, there exists $x^*\in X$ such that
	$x_n\to x^*$ as $n\to\infty$.	It remains to prove that $x^*\in Tx^*$. We consider the following 3 cases:
	assumptions in the statement.
	
\paragraph*{\textbf{Case A}}
	Suppose that $T$ is $h$-upper semicontinuous. Since
	$x_{n+1}\in Tx_n$ and $x_n\to x^*$, we have
	$d(x^*,x_{n+1})\to0.$
	By the $h$-upper semicontinuity of $T$,
	$\delta(Tx_n,Tx^*)\to0.$
	Using the $b$-metric inequality and Lemma~\ref{properties}\ref{14c},
	we obtain $$
	D(x^*,Tx^*)
	\leq
	s\left(d(x^*,x_{n+1})+D(x_{n+1},Tx^*)\right)
	\leq
	s\left(d(x^*,x_{n+1})+\delta(Tx_n,Tx^*)\right).$$
	Letting $n\to\infty$ gives $
	D(x^*,Tx^*)=0.$ Hence $x^*\in Tx^*$.
	
\paragraph*{\textbf{Case B}} Suppose  that $\displaystyle 	s\limsup_{t\to0^+}\left(c(t)+e(t)\right)<1$. Then, let $a_n=a(d(x_n,x^*)), ...,f_n=f(d(x_n,x^*))$. By Lemma \ref{properties}\ref{14c}, inequality \eqref{A3} and $b$-metric inequality, we get 
	\begin{align*}
		D(x_{n+1},Tx^*)&\leq H(Tx_n,Tx^*)\\
		&\leq a_nd(x_n,x^*)+b_nD(x_n,Tx_n)+c_nD(x^*,Tx^*)+e_nD(x_n,Tx^*)+f_nD(x^*,Tx_n)\nonumber\\
		&\leq  a_nd(x_n,x^*)+b_nd_{n+1}+c_ns[d(x^*,x_{n+1})+D(x_{n +1},Tx^*)]+se_n[d_{n+1}+D(x_{n+1},Tx^*)]\\&\qquad+f_nd(x^*,x_{n+1}).
	\end{align*}
	Therefore, we have 
	\begin{align}
		\label{D}
		(1-s(c_n+e_n))D(x_{n+1},Tx^*)
		&\leq a_nd(x_n,x^*)+(b_n+se_n)d_{n+1}+c_n sd(x^*,x_{n+1}) +f_nd(x^*,x_{n+1}).
	\end{align}
	By the assumption  $\displaystyle 	s\limsup_{t\to0^+}\left(c(t)+e(t)\right)<1$, there exist $\epsilon>0$ and $N\in \mathbb{N}$ such that $$1-s(c_n+e_n)\geq \epsilon, \text{ for all }n\geq N.$$  Since $d(x_n,x^*)\to 0,d_{n+1}\to 0,~ d(x^*,x_{n+1})\to0$ and $0\leq a_n,b_n,c_n,e_n,f_n<1$, the right-hand side of \eqref{D} tends to zero as $n\to \infty$.   Hence
	\begin{align*}
		\lim_{n \to \infty}D(x_{n+1},Tx^*)= 0.
	\end{align*}
	By the $b$-metric inequality, we have 
	$$D(x^*,Tx^*)\leq s(d(x^*,x_{n+1})+D(x_{n+1},Tx^*)).$$
	Taking limit as $n\to\infty$, we obtain $D(x^*,Tx^*)=0$. Hence, $x^*\in Tx^*.$ 
	
	\paragraph*{\textbf{Case C}}
	Suppose that $	\displaystyle\limsup_{t\to0^+}\left(b(t)+f(t)\right)<1.$
	The proof is analogous to Case B, with the roles of $x_n$ and $x^*$
	interchanged in the terms involving $b$ and $f$. In particular, one
	obtains
	$
	\left(1-s(b_n+f_n)\right)D(x_{n+1},Tx^*)\to0,
	$
	while the assumed condition guarantees that $
	1-s(b_n+f_n)$
	is bounded away from zero for all sufficiently large $n$. Hence
	$
	D(x_{n+1},Tx^*)\to0.$
	Consequently,
	$
	D(x^*,Tx^*)=0,$
	and therefore $x^*\in Tx^*$.
	
	Thus, in each of the three cases, $x^*$ is a fixed point of $T$.
	Hence $T$ has a fixed point.
\end{proof}

\begin{remark}
	\label{remarkpiecewise}
	The condition
	$a(t)+b(t)+c(t)+2s\min\{e(t),f(t)\}<1$
	naturally leads to two cases in the proof of Theorem~\ref{Main2}.
	Indeed, if $e(t)\leq f(t)$, then
	$$
0\leq 	\frac{a(t)+b(t)+se(t)}
	{1-c(t)-se(t)}<1,$$
	whereas if $f(t)<e(t)$, then
	$$
0\leq 	\frac{a(t)+c(t)+sf(t)}
	{1-b(t)-sf(t)}<1.$$
	Thus, the  contraction factor corresponding to the active case may
	equivalently be represented by the piecewise contraction factor \eqref{gamma}. In particular, $0\leq \gamma(t)<1$ identifies the contraction
	factor associated with the relevant case.  Thus, $\gamma(t)$ is well defined for every $t\geq0$. 
\end{remark}

\begin{remark}
	We note an apparent algebraic discrepancy in the contraction constant stated in \cite[Theorem 3.2]{Chaib2026}. The contraction constant is defined by $
	\lambda=\min\{\lambda_1,\lambda_2\},$
	where
	$$
	\lambda_1=\frac{a+b+sd}{1-c-sd}, \text{ and }
	\lambda_2=\frac{a+b+se}{1-c-se}.$$
	However, the estimate corresponding to the second case yields
	$$
	\lambda_2=\frac{a+c+se}{1-b-se}.$$
	Consequently, the contraction constant should be expressed as
	$$
	\lambda_*=\min\left\{\frac{a+b+sd}{1-c-sd},\frac{a+c+se}{1-b-se}\right\}.$$
	
	 Moreover, the choose of mimimum do not ensure positivity of both the denominators in the resulting corrected expression $\lambda_*$. For example, the principal assumption
	$
	a+b+c+2s\min\{d,e\}<1$ in \cite[Theorem 3.2]{Chaib2026}
	does not imply
	$c+sd<1$, whenever $d\neq e.$
	For instance,  take $s=2$, and 
	$$a=0,\quad b=\frac14,\quad c=\frac16,\quad d=\frac12,\quad e=\frac1{16}.$$
	Then
$$	a+b+c+2s\min\{d,e\}
	=
	\frac14+\frac16+4\left(\frac1{16}\right)
	=
	\frac23<1.$$
	
	On the other hand,
	$$1-c-sd=-\frac{1}{6}<0 \text{ and } 1-b-se=\frac{5}{8}.$$
	Consequently,
	$$
	\lambda_*=
	\min\left\{-\frac{15}{2},\frac7{15}\right\}=
	-\frac{15}{2}<0.$$
	Thus, $a+b+c+2s\min\{d,e\}<1$ does not necessarily imply that $\lambda_*$ is a positive contraction constant. This example demonstrates that correcting the algebraic form of the contraction constant alone is insufficient; additional conditions are required to guarantee the positivity of the corresponding denominators. Without such conditions, the proposed expression need not define a contraction constant in $[0,1)$.

We address this issue, in Theorem~\ref{Main2} and Remark~\ref{remarkpiecewise} by introducing the piecewise contraction factor $\gamma(t)$ defined in \eqref{gamma}. The denominators are positive and  $
\gamma(t)\in[0,1), \text{ for all } t\geq 0.$ In the example above, $\gamma=\frac{7}{15}<1.$
\end{remark}

Under the additional assumption that the piecewise contraction factor in \eqref{gamma} is nonincreasing, the supremum condition in Theorem \ref{Main2} may be omitted. Consequently, we obtain the following corollary.
\begin{corollary}
	\label{Main3}
		Let $(X,d)$ be a complete $b$-metric space and  let $T: X \to \mathcal{C}(X)$ be a Hardy-Rogers type multivalued mapping. Suppose the function $\gamma(t)$  in \eqref{gamma} is nonincreasing. Then $T$ has a fixed point in $X$ if any of the following conditions are
	satisfied
	\begin{enumerate}[label=(\alph*)]
		\item $T$ is $h$-upper semicontinuous;
		\item $\displaystyle s	\limsup_{t\to0^+}\left(c(t)+e(t)\right)<1$;
	\item $\displaystyle s	\limsup_{t\to0^+}\left(b(t)+f(t)\right)<1$. 
	\end{enumerate}
\end{corollary}

\begin{proof} 
	Since $\gamma(t)$ is nonincreasing, we have $\gamma(t)\leq \gamma(0)<1$ for all $t\geq 0.$ Hence
	$$\sup_{t \geq 0}  \gamma(t)=\gamma(0)<1$$ and Theorem \ref{Main2} applies. Therefore,  $x^*\in Tx^*$.
\end{proof}

\begin{corollary}[Reich Type  Multivalued  Mappings]
	\label{Reich}
Let $(X,d)$ be a complete $b$-metric space and $T: X \to \mathcal{C}(X)$ a multivalued mapping. Suppose there exist nonnegative functions $a,b,c:[0,\infty)\to [0,1)$ with 
\begin{eqnarray}
	\gamma^*=\sup_{t \geq 0}
\min\left\{
\frac{a(t)+b(t)}{1-c(t)},
\frac{a(t)+c(t)}{1-b(t)}
\right\}<1.
\end{eqnarray}
 such that
\begin{eqnarray}
	\label{Equation Reich}
	H(Tx,Ty)\leq a(d(x,y))d(x,y)+b(d(x,y))D(x,Tx)+c(d(x,y))D(y,Ty).
\end{eqnarray}
Then T has a fixed point in $X$ if any of the following conditions are
satisfied:
\begin{enumerate}[label=(\alph*)]
	\item $T$ is $h$-upper semicontinuous;
				\item $\displaystyle s\limsup_{t \to 0^+}c(t)<1$;
	\item  $\displaystyle s\limsup_{t \to 0^+}b(t)<1$.
\end{enumerate}
\end{corollary}

\begin{proof}
	Choose $x_0\in X$ and $x_1\in Tx_0$. If $x_1\in Tx_1$, then
$x_1$ is a fixed point of $T$ and the proof is complete. Hence, assume
that $x_1\notin Tx_1$. We construct a sequence $\{x_n\}$ inductively
such that $x_{n+1}\in Tx_n,$ for all $ n\geq1.$
By Lemma~\ref{lemmaalpha}, for every $\eta>1$ and $n\geq1$, we can
choose $x_{n+1}\in Tx_n$ such that \eqref{A} holds. 
	Substituting $x=x_{n-1}$ and $y=x_n$ in \eqref{Equation Reich}, we get
	\begin{align*}
	H(Tx_{n-1},Tx_n)
		&\leq a(d_n)d_n+b(d_n)D(x_{n-1},Tx_{n-1})
		+c(d_n)D(x_n,Tx_n)\\
		&\leq (a(d_n)+b(d_n))d_n
		+c(d_n)H(Tx_{n-1},Tx_n).
	\end{align*}
	Consequently,
\begin{equation}
	\label{R1}
		H(Tx_{n-1},Tx_n)
	\leq
	\frac{a(d_n)+b(d_n)}
	{1-c(d_n)}\,d_n.
\end{equation}
	On the other hand, using the symmetry of the Hausdorff metric and
	applying the contractive condition \eqref{Equation Reich} with $x=x_n$ and $y=x_{n-1}$,
	we obtain
	\begin{align*}
H(Tx_n,Tx_{n-1})
		&\leq a(d_n)d_n+b(d_n)D(x_n,Tx_n)
		+c(d_n)D(x_{n-1},Tx_{n-1})\\
		&\leq \bigl(a(d_n)+c(d_n)\bigr)d_n
		+b(d_n)d_{n+1}.
	\end{align*}
	Hence,
\begin{equation}
	\label{R2}
	H(Tx_{n-1},Tx_n)
	\leq
	\frac{a(d_n)+c(d_n)}
	{1-b(d_n)}\,d_n.
\end{equation}
 Define
	$$
	\tau(t)=
	\min\left\{
	\frac{a(t)+b(t)}{1-c(t)},
	\frac{a(t)+c(t)}{1-b(t)}
	\right\}<1.$$
It follows from \eqref{R1} and \eqref{R2} that
\begin{equation}
	\label{R3}
	H(Tx_{n-1},Tx_n)\leq \tau(d_n)d_n.
\end{equation}
Since $\displaystyle\gamma^*=\sup_{t\geq0}\tau(t)<1,$
we may choose $\eta>1$ sufficiently close to $1$ such that
$
\beta=\eta\gamma^*<1.$
Combining \eqref{A} and \eqref{R3}, we obtain
$$
d_{n+1}
\leq \eta H(Tx_{n-1},Tx_n)
\leq \eta\tau(d_n)d_n
\leq \eta\gamma^*d_n
=\beta d_n.$$
Consequently, $
d_{n+1}\leq\beta d_n, \text{ for } 0\leq\beta<1.$
It now follows from Lemma~\ref{lem} that $\{x_n\}$ is a Cauchy sequence. Since $(X,d)$ is complete, there exists $x^*\in X$ such that
$x_n\to x^* \text{ as }n\to\infty.$ The proof of the remaining assertions proceeds in the same manner as in Cases A, B, and C of Theorem~\ref{Main2}, upon setting $e(t)=f(t)=0$. Therefore, $x^*\in Tx^*.$
\end{proof}

\begin{corollary}[Kannan Type Multivalued Mappings]
	\label{Kannan}
	Let $(X,d)$ be a complete $b$-metric space and $T: X \to \mathcal{C}(X)$ a multivalued mapping. Suppose there exists a nonnegative function
	$
	\alpha:[0,\infty)\to[0,\tfrac12)$
	such that $\displaystyle
	\sup_{t\geq0}\alpha(t)<\frac12$ such that $\forall x,y\in X$
	\begin{eqnarray*}
		H(Tx,Ty)\leq \alpha(d(x,y))(D(x,Tx)+D(y,Ty)).
	\end{eqnarray*}
	Then $T$ has a fixed point in $X$  if any of the following conditions are
	satisfied
	\begin{enumerate}[label=(\alph*)]
		\item $T$ is $h$-upper semicontinuous;
		\item $\displaystyle s\limsup_{t \to 0^+}\alpha(t)<1$;
		\item  $s<2$.
	\end{enumerate} 
\begin{proof}
Setting $a(t)=e(t)=f(t)=0$ and $b(t)=c(t)=\alpha(t)$ in Theorem~\ref{Main2} yields
$$
\gamma(t)=\frac{\alpha(t)}{1-\alpha(t)}.$$
Moreover,
$$
\sup_{t\geq0}\alpha(t)<\frac12
\quad\Longleftrightarrow\quad
\sup_{t\geq0}\gamma(t)<1.$$
Thus, the corresponding condition on $\gamma$ in Theorem~\ref{Main2} is satisfied. Furthermore, if $s<2$, then
$$
s\limsup_{t\to0^+}\alpha(t)
\leq s\sup_{t\geq0}\alpha(t)
<\frac{s}{2}<1.$$
Hence, all the hypotheses of Theorem~\ref{Main2} are fulfilled, and consequently, $T$ has a fixed point.
\end{proof}
\end{corollary}

	\begin{corollary}[Chatterjea Type  Multivalued  Mappings]
		\label{Chattejea Type}
	Let $(X,d)$ be a complete $b$-metric space and $T: X \to \mathcal{C}(X)$ a multivalued mapping. Suppose there exists nonnegative function $\alpha$ with $\displaystyle \sup_{t\geq 0} \alpha (t) < \frac{1}{2s}$ such that for all $ x,y\in X$ 
	\begin{eqnarray*}
		\label{Chatterjea}
		H(Tx,Ty)\leq \alpha(d(x,y))(D(x,Ty)+D(y,Tx)). 
	\end{eqnarray*}
 Then $T$ has a fixed point in $X$ provided either $T$ is
 $h$-upper semicontinuous or $ \displaystyle \limsup_{t\to0^+}\alpha(t)<\frac{1}{s}.$
\end{corollary}

\begin{proof}
	
	Setting $a(t)=b(t)=c(t)=0$ and $e(t)=f(t)=\alpha(t)$ in Theorem~\ref{Main2} yields
	$$
\displaystyle	\gamma(t)=\frac{s\alpha(t)}{1-s\alpha(t)}.$$
	Moreover,
	$$
	\sup_{t\geq0}\alpha(t)<\frac{1}{2s}
	\quad\Longleftrightarrow\quad
	\sup_{t\geq0}\gamma(t)<1.$$
	Thus, the corresponding condition on $\gamma$ in Theorem~\ref{Main2} is satisfied. Consequently, there exists $x^*\in X$ such that $x^*\in Tx^*.$
 Thus, $x^*$ is a fixed point of $T$.
\end{proof}

\section{Strict Fixed Points for Hardy-Rogers Type Multivalued Mapppings}
In this section, we examine strict fixed points of Hardy–Rogers type multivalued mappings and derive conditions ensuring their uniqueness. In particular, we show that the existence of a strict fixed point with an additional condition $a(t)+e(t)+f(t)<1$ implies that the fixed point set and the strict fixed point set coincide and reduce to a singleton. 

\begin{definition}[Nadler \cite{Nadler}]
	A point $x \in X$ is called a \textit{strict fixed point} of a multivalued mapping $T: X \to \mathcal{C}(X)$  if
	$Tx^* = \{x^*\}.$ The set of strict
	fixed points is denoted by $$\text{SFix}(T) = \{x^* \in X \mid \{x^*\} = Tx^*\}\subseteq \text{Fix }(T).$$
\end{definition}
\begin{remark}
	This is a stronger condition than just $x^* \in Tx^*$. It means that the image of $x$ under $T$ is exactly the singleton set containing $x$ itself.
\end{remark}

\begin{theorem}
	\label{MainSFix}
	Suppose that the conditions of Theorem~\ref{Main2} are satisfied. If $x^*\in \text{SFix}(T)\neq\emptyset$ and $a(t)+e(t)+f(t)<1$, then
	$\text{SFix}(T)=\text{Fix}(T)=\{x^*\}.$
\end{theorem}

\begin{proof}
Since $x^*\in \text{SFix}(T)$, we have $Tx^*=\{x^*\}$ and hence
$x^*\in \text{Fix}(T)$. Thus $\text{SFix}(T)\subseteq \text{Fix}(T).$
Let $w\in \text{Fix}(T)$. Suppose, to the contrary, that $w\neq x^*$.
Set $
a_+=a(d(x^*,w)),~ b_+=b(d(x^*,w)),~
c_+=c(d(x^*,w)),~ e_+=e(d(x^*,w)),~ f_+=f(d(x^*,w)).$
Since $Tx^*=\{x^*\}$ and $w\in Tw$, Lemma~\ref{properties}\ref{14c}
yields
\begin{align*}
	d(x^*,w) &=D(Tx^*,w)\leq H(Tx^*, Tw) \\
	&\leq a_+d(x^*,w)+ b_+D(x^*,Tx^*)+c_+D(w,Tw)+e_+D(x^*,Tw)+f_+D(w,Tx^*)\\
	&\leq (a_++e_++f_+)d(x^*,w).
\end{align*}
Since $a(t)+e(t)+f(t)<1$ for all $t\geq 0$, we have  $d(x^*,w)=0.$ Thus $x^*=w$, a contradiction. Therefore,
$\text{Fix}(T)=\{x^*\}.$
Since $\text{SFix}(T)\subseteq \text{Fix}(T)$ and $x^*\in \text{SFix}(T)$, we conclude that
$
\text{SFix}(T)=\text{Fix}(T)=\{x^*\}.$
\end{proof}

	\begin{corollary}
	\label{Hardyconstant}
	Let $(X,d)$ be a complete $b$-metric space and $T: X \to \mathcal{C}(X)$ be a  mapping. Suppose there exist nonnegative numbers $a,b,c,e,f$ satisfying $a+b+c+2s\min\{e,f\}<1$ with
	\begin{eqnarray}
		\label{gammac}
		\gamma=\begin{cases}	\displaystyle	\frac{a+b+se}	{1-c-se},	& e\leq f,\\[2ex]	\displaystyle	\frac{a+c+sf}{1-b-sf},	& f<e.\end{cases}
	\end{eqnarray}
	
	such that for all $x,y\in X$
	\begin{align}
		\label{A4}
		H(Tx,Ty)&\leq a d(x,y)+bD(x,Tx)+ cD(y,Ty)+eD(x,Ty)+fD(y,Tx).
	\end{align}	
	Then T has a fixed point if any of the following conditions are
	satisfied:
	\begin{enumerate}[label=(\alph*)]
		\item $T$ is $h$-upper semicontinuous;
		\item $s(c+e)<1$;
		\item  $s(b+f)<1$.
	\end{enumerate}
If $x^*\in \text{SFix}(T)\neq\emptyset$ and $a+e+f<1$, then
$\text{SFix}(T)=\text{Fix}(T)=\{x^*\}.$
\end{corollary}
\begin{proof}
    Apply Theorem \ref{Main2}  with constant coefficient functions $a(t)\equiv a, b(t)\equiv b, c(t)\equiv c,e(t)\equiv e,f(t)\equiv f$ to obtain existence. Since  $x^*\in SFix(T)$ and $a+e+f<1$, apply Theorem Theorem \ref{MainSFix}  to obtain
$SFix(T)=Fix(T)=\{x^*\}$.

\end{proof}
\begin{remark}
The hypotheses of Corollary~\ref{Hardyconstant} are weaker than the conditions in \cite[Theorem 3.6]{Chaib2026}. Indeed, for $s\geq1$,
	$$
	s(c+e)\leq sc+s^2e,
\text{ and }
	s(b+f)\leq sb+s^2f.
	$$
	Consequently,
	$$
	sc+s^2e<1
	\Longrightarrow
	s(c+e)<1,
	\text{ and }
	sb+s^2f<1
	\Longrightarrow
	s(b+f)<1.
	$$
	Thus, the conditions $s(c+e)<1$ and $s(b+f)<1$ are no stronger than the corresponding conditions involving $s^2e$ and $s^2f$. When $s>1$ and $e>0$ (respectively, $f>0$), the corresponding inequality is strict.
	
	Moreover, Corollary \ref{Hardyconstant} extends the framework of \cite[Corollary 4.8]{Li} by replacing the condition
	$$
	a+b+c+e+f<\frac{1}{s}$$
	with the alternative requirement
	$$
	a+b+c+2s\min\{e,f\}<1 \text{ and } \eqref{gammac}$$
	together with any of the following conditions: $T$ is  $h$-upper semicontinuous, 
	$
	s(c+e)<1
	\text{ or }
	s(b+f)<1.$
Our hypotheses admit parameter configurations that are not admissible under the hypotheses of \cite[Theorem 3.6]{Chaib2026} and \cite[Corollary 4.8]{Li}, as demonstrated by Examples~\ref{Example 1} and~\ref{Example 2}.
\end{remark}

\begin{example}
	\label{Example 1}
	Let $X = [0,1]$ and $d(x,y) = |x-y|^2$. Clearly, $(X,d)$ is a complete $b$-metric space with constant $s=2$.  Define $T: X \to \mathcal{C}(X)$ by
	
	\begin{equation*}
		Tx = 
		\begin{cases}
			\left\{\dfrac{x}{3}\right\}, & x \in [0,1), \\[8pt]
			\left\{0, \dfrac{1}{3}\right\}, & x = 1.
		\end{cases}
	\end{equation*}
	Choose constant functions $a = 0,~b = 0.1,~c= 0.05, ~e= 0.25,~ f= 0.2.$
	Therefore, $\gamma=0.9<1$ and  $$a+ b + c + 2s\min\{e, f\}=0.95<1.$$
	Let $\{x_n\} \subset X$ satisfy $d(x_n, x) = |x_n - x|^2 \to  0.$ If $x \in [0,1)$, then, for all sufficiently large $n$, $x_n < 1$, and hence
	\begin{equation*}
		\delta(Tx_n, Tx) = d\left(\frac{x_n}{3}, \frac{x}{3}\right)
		= \frac{1}{9}|x_n - x|^2 \to  0.
	\end{equation*}
	If $x = 1$, then either $x_n\to 1^-$ or $x_n=1$. If $x_n = 1$ for infinitely many $n$, then $\delta(Tx_n, T1) = 0$ along that subsequence. If $x_n \to 1^-$, then we have 
	\begin{align*}
		\delta(Tx_n, T1)
		= d\left(\frac{x_n}{3}, \left\{0, \frac{1}{3}\right\}\right)
		= \min\left\{ \frac{x_n^2}{9}, \frac{(1-x_n)^2}{9} \right\} 
		\leq \frac{(1-x_n)^2}{9} \to  0.
	\end{align*}
	Thus, $\displaystyle \lim_{n \to \infty} \delta(Tx_n, Tx) = 0.$
	By Definition~\ref{Continuous}, $T$ is $h$-upper semicontinuous on $X$.
	
	Next, we show the inequality \eqref{A3} holds for all $x,y\in X$. Consider the following cases:
	\subsection*{Case 1: $x, y \in [0,1)$}
	Then, by inequality \eqref{A3}, we have 
	$$
	\frac{1}{9}|x-y|^2 \leq \frac{1}{10}\cdot\frac{4x^2}{9} + \frac{1}{20}\cdot\frac{4y^2}{9} + \frac{1}{4}\left|x - \frac{y}{3}\right|^2 + \frac{1}{5}\left|y - \frac{x}{3}\right|^2.
	$$
	which is equivalent to
	$0 \leq 37x^2 - 14xy + 25y^2.$ The quadratic form $Q(x,y)=37x^2 - 14xy + 25y^2$ has discriminant
	$
	\Delta = (-14)^2 - 4(37)(25) = 196 - 3700 = -3504 < 0,$
	and the leading coefficient $37 > 0$. Therefore the inequality holds for all $x, y \in [0,1)$.
	
	\subsection*{Case 2: $x = 1$ and $y \in [0,1)$}
	Then, by inequality \eqref{A3}, we have 
	\begin{eqnarray*}
		\label{max}
		\max\left\{ \frac{y^2}{9}, \frac{(1-y)^2}{9} \right\}
		\leq
		\frac{1}{10}\cdot\frac{4}{9} + \frac{1}{20}\cdot\frac{4y^2}{9} + \frac{1}{4}\left(1 - \frac{y}{3}\right)^2 + \frac{1}{5}\min\left\{ y^2, \left(y - \frac{1}{3}\right)^2 \right\}.
	\end{eqnarray*}
	Thus, 
	$$\max_{0\le y<1}\operatorname{LHS}(y)=\frac19\text{ and }
	\min_{0\le y<1}\operatorname{RHS}(y)=\frac{17}{75}.$$
	
	Since $\frac19<\frac{17}{75}$, the inequality holds for all $y\in[0,1)$.	
	
	\subsection*{Case 3: $x \in [0,1)$ and $y = 1$ }
	The argument is symmetric to Case 2 and is therefore omitted.
	\subsection*{Case 4: $x = 1$, $y=1$} This case is trivial since $H(Tx,Ty)=0.$
	
Thus, inequality \eqref{A3} holds for all $x,y\in X$. Consequently, Theorem~\ref{Main2} implies that $T$ has a fixed point. Since this fixed point is strict, Theorem~\ref{MainSFix} yields $\mathrm{SFix}(T)=\mathrm{Fix}(T)=\{0\}$.
\end{example}

\begin{remark}
	Observe that \cite[Theorem 3.6]{Chaib2026} is not applicable in this example, since
$	\lambda=\min\left\{\frac{4}{3},0.9\right\}$
	does not satisfy the required condition for a contraction constant, as one of the entries exceeds $1$. Moreover,
	$$
	sc+s^2e=1.1>1,\qquad sb+s^2f=1,
	$$
	so the corresponding hypotheses of \cite[Theorem 3.6]{Chaib2026} are not satisfied. Likewise, \cite[Corollary 4.8]{Li} cannot be applied to establish the existence of a strict fixed point, since
	$$a+b+c+e+f=0.6>\frac{1}{2}=\frac{1}{s}.$$
\end{remark}

\begin{example}
	\label{Example 2}
	Let $X=[0,\infty)$. Define 
	$M_1=\mathbb{Q}\cap[0,\infty)$ and $M_2=[0,\infty)\setminus\mathbb{Q}.$
	Clearly, both $M_1$ and $M_2$ are dense in $X$. Let 
	$$
	d(x,y)=|x-y|^2,~~~ x,y\in X.$$
	Then $(X,d)$ is a complete $b$-metric space with coefficient $s=2$.
	Define $T:X\to C(X)$ by
	\begin{eqnarray*}
		T(x)=
		\begin{cases}
			\{qx,rx\}, & x\in M_1,\\[2mm]
			\{rx\}, & x\in M_2,
		\end{cases}
	\end{eqnarray*}
	where $q=0.2$ and $r=0.15.$
	
	We first show that $T$ is not $h$-upper semicontinuous. Let $x\in M_2$ with $x>0$. Since $M_1$ is dense in $X$, there exists a sequence $\{x_n\}\subset M_1$ such that $x_n\to x$. Then
	$$T(x_n)=\{qx_n,rx_n\},~~T(x)=\{rx\}.$$
	Consequently,
	$$\delta(Tx_n,Tx)=\max\left\{d(qx_n,rx),d(rx_n,rx)\right\}.$$
	Since
	$d(qx_n,rx)=|qx_n-rx|^2\to	(q-r)^2x^2$
	and $
	d(rx_n,rx)=r^2|x_n-x|^2\to 0,$
	we obtain
	\[
	\lim_{n\to\infty}\delta(Tx_n,Tx)
	=
	(q-r)^2x^2>0.
	\]
	Thus, $T$ is not $h$-upper semicontinuous at every $x\in M_2\setminus\{0\}$.
	
	We next verify the contractive condition. Let
	$$
	a=\frac1{20},~~
	b=\frac1{10},~~
	c=\frac1{40},~~
	e=\frac14,~~
	f=\frac15.$$
	Since $s=2$, we have $\gamma=0.95<1$ and 
	$
	a+b+c+2s\min\{e,f\}=0.975<1.$
	We consider three cases.
	
	\subsection*{Case 1:} $x,y\in M_1$. Since $q\ge r$,
	$H(Tx,Ty)=\max\{q^2|x-y|^2,r^2|x-y|^2\}	=	q^2|x-y|^2.$
	Furthermore,
	\begin{align*}
		D(x,Tx)&=(1-q)^2x^2,~~
		D(y,Ty)=(1-q)^2y^2,\\
		D(x,Ty)&=\min\{|x-qy|^2,|x-ry|^2\},~~
		D(y,Tx)=\min\{|y-qx|^2,|y-rx|^2\}.
	\end{align*}
	By \eqref{A3}, we obtain
	\begin{equation}
		\label{qq}
		q^2|x-y|^2	\leq	\frac1{20}|x-y|^2	+\frac1{10}(1-q)^2x^2	+\frac1{40}(1-q)^2y^2+\frac{1}{4}D(x,Ty)+\frac{1}{5}D(y,Tx).
	\end{equation}
	Since $q^2=\frac{1}{25}<\frac{1}{20}$, then the inequality \eqref{qq} holds for all $x,y\in M_1$.
	
	\subsection*{Case 2:} $x,y\in M_2$. In this case, $T(x)=\{rx\}, T(y)=\{ry\}.$ Hence $ 
	H(Tx,Ty)=r^2|x-y|^2.$
	Also, 
	\begin{align*}
		D(x,Tx)&=(1-r)^2x^2,~~
		D(y,Ty)=(1-r)^2y^2,\\
		D(x, Ty) &= d(x, ry) = |x - ry|^2,~~
		D(y, Tx) = d(y, rx) = |y - rx|^2.
	\end{align*}
	Therefore,
	\begin{equation}
		\label{rr}
		r^2|x-y|^2\leq \frac1{20}|x-y|^2+\frac1{10}(1-r)^2x^2 +\frac1{40}(1-r)^2y^2+\frac{1}{4}|x - ry|^2+\frac{1}{5}|y - rx|^2.
	\end{equation}
	
	Since $r^2=\frac{9}{400}<\frac{1}{20}$, then the inequality \eqref{rr} holds for all $x,y\in M_2$.
	
	\subsection*{Case 3:} $x\in M_1$ and $y\in M_2$. Since $q\ge r$, we have 
	\begin{align*}
		H(Tx,Ty)&=\max\{|qx-ry|^2,|rx-ry|^2\}	=	|qx-ry|^2,\\
		D(x,Tx)&=(1-q)^2x^2,~~	D(y,Ty)=(1-r)^2y^2,\\
		D(x,Ty)&=|x-ry|^2, ~~D(y,Tx)=\min\{|y-qx|^2,|y-rx|^2\}.
	\end{align*}
	Therefore, $$
	|qx-ry|^2
	\le
	\frac1{20}|x-y|^2
	+\frac1{10}(1-q)^2x^2
	+\frac1{40}(1-r)^2y^2
	+\frac14|x-ry|^2
	+\frac15D(y,Tx).$$
	In Case 3, the inequality splits into two subcases depending on the minimum in $D(y, Tx)$.
	
	\textbf{Case 3a:} When $|y - qx| \leq |y - rx|$, we have $D(y, Tx) = |y - qx|^2$. The difference between the right-hand side and the left-hand side is the quadratic form
	$$
	Q_1(x,y) = \frac{83}{250}x^2 - \frac{39}{200}xy + \frac{4019}{16000}y^2,$$
	and its discriminant is
	$$
	\Delta_1 = B_1^2 - 4A_1C_1
	= \left( - \frac{39}{200} \right)^2
	- 4\left( \frac{83}{250} \right)\left( \frac{4019}{16000} \right)
	=  -\frac{36944}{125000} < 0.$$
	
	\textbf{Case 3b:} When $|y - rx| \le |y - qx|$, we have $D(y, Tx) = |y - rx|^2$. The difference is the quadratic form
	$$
	Q_2(x,y) = \frac{657}{2000}x^2 - \frac{7}{40}xy + \frac{4019}{16000}y^2,$$
	and its discriminant is
	$$
	\Delta_2 = B_2^2 - 4A_2C_2
	= \left( -\frac{7}{40} \right)^2
	- 4\left( \frac{657}{2000} \right)\left( \frac{4019}{16000} \right)
	= -\frac{2395483}{8000000} < 0.$$
	Since both discriminants are negative and the leading coefficients are positive, both quadratic forms are positive definite. 
	Hence $Q_i(x,y)>0$ for every $(x,y)\ne(0,0)$. Since the inequality is trivial at $(0,0)$, the required estimate follows.

	\subsection*{Case 4:} $x\in M_2$ and $y\in M_1$. This case follows by symmetry from Case 3.

	Thus, the  inequality \eqref{A3} is satisfied for all $x, y \in X$. 	Consequently, Theorem~\ref{Main2} implies that $T$ has a fixed point. Since this fixed point is strict, Theorem~\ref{MainSFix} yields $\mathrm{SFix}(T)=\mathrm{Fix}(T)=\{0\}$.
\end{example}
\begin{remark}
	\label{remarkexample2}
	In Example  \ref{Example 2},
	$
	\lambda
	=
	\min\left\{\frac{26}{19},\frac{19}{20}\right\}
	=\frac{19}{20}<1.$
	However, the first quantity entering is greater than 1
	and therefore cannot serve as a contraction constant. Although, the formal minimum is less than $\lambda<1$, the use of $\lambda$ as a contraction constant is not justified unless each of the quantities entering the minimum is known to belong to $[0,1)$. This illustrates a deficiency in the formulation based solely on the minimum of the two expressions.
	Moreover,
	$
	sc+s^2e=1.05>1,
	sb+s^2f=1,$
	and
	$
	a+b+c+e+f
	=\frac{5}{8}
	>\frac{1}{2}
	=\frac{1}{s}.$
	Hence, Example \ref{Example 2} lies outside the scope of  \cite[Theorem 3.6]{Chaib2026} and \cite[Corollary 4.8]{Li}. More importantly, it demonstrates that taking the minimum of two expressions does not, by itself, guarantee that both underlying contraction estimates are valid contraction bounds.
\end{remark}

\begin{corollary}
	\label{A11}
	Let $(X,d)$ be a complete $b$-metric space and $T: X \to \mathcal{C}(X)$ be a multivalued mapping. Suppose there exist nonnegative numbers $a_1, a_2, a_3\in [0,1)$ satisfying $a_1 + 2a_2 + 2sa_3 < 1$ such that for all $x, y \in X$,
	$$
	H(Tx, Ty) \leq a_1\, d(x,y) + a_2(D(x,Tx) + D(y,Ty)) + a_3(D(x,Ty) + D(y,Tx)).
	$$
 Then $T$ has a fixed point. Furthermore,  if either $T$ is $h$-upper semicontinuous or $sa_2 + sa_3 < 1$ and  $\text{SFix}(T)\neq \emptyset$ with $x^*\in \text{SFix}(T)$, then $\text{SFix}(T) = \text{Fix}(T) = \{x^*\}.$
\end{corollary}
\begin{proof}
	Take  $a=a_1$, $b=c=a_2$ and $e=f=a_3$ in Corollary \ref{Hardyconstant}.
\end{proof}
\begin{remark}
Corollary~\ref{A11} extends and generalizes the results of
\cite[Corollary 3.7]{Chaib2026} and \cite[Theorem 2.2]{Chifu2017} by
substantially relaxing their contractive assumptions. In particular, the condition $sa_2+s^{2}a_3<1$ is replaced by the less restrictive condition $sa_2+sa_3<1$. These conditions coincide when $s=1$; however, for $s>1$ the latter is strictly weaker provided that $a_3>0$.
\end{remark}

\section{Fixed Points and Strict Fixed Points of Iterates}
In this section, we examine the relationship between the fixed points of a multivalued mapping and those of its finite iterates. In particular, we show that if a suitable finite iterate $T^{p}$ satisfies the hypotheses of Theorem~\ref{MainSFix}, then, provided that $\mathrm{SFix}(T)\neq\emptyset$, its unique strict fixed point is also the unique strict fixed point of $T$. This result provides fixed point conclusions for mappings whose contractive behavior emerges only after a finite number of iterations. We further show that the assumption $\mathrm{SFix}(T)\neq\emptyset$ cannot, in general, be omitted: although every strict fixed point of $T$ is necessarily a strict fixed point of each iterate $T^{p}$, the converse implication fails for multivalued mappings. This distinction is demonstrated by the example given in Remark~\ref{strictTp}.

\begin{definition}
	Let $T:X\to C(X)$ be a multivalued mapping. Its iterates are defined
recursively by
$$
T^1x=Tx \text{ and }T^{p}x=\bigcup_{y\in T^{p-1}x}Ty,\quad p\geq2.
$$
Whenever needed, we assume that $T^p:X\to C(X)$.
\end{definition}
\begin{example}
	Let $X=\mathbb{R}$ and define $T:X\to C(X)$ by
	$$
	T(x)=
	\begin{cases}
		[0,1], & x=2,\\
		\{1\}, & x=0,\\
		\{x\}, & 0<x\leq1,\\
		\{0\}, & \text{otherwise}.
	\end{cases}$$
	Then $T(x)$ is nonempty and closed for every $x\in X$, so
	$T:X\to C(X)$. However,
	$$
	T^2(2)
	=\bigcup_{y\in T(2)}T(y)
	=\bigcup_{y\in[0,1]}T(y)
	=\{1\}\cup\bigcup_{0<y\leq1}\{y\}
	=(0,1],$$
	which is not closed in $\mathbb{R}$. Therefore,
	$
	T^2:X\not\longrightarrow C(X).$
\end{example}

\begin{theorem}
		\label{Main6}
	Let $(X,d)$ be a complete $b$-metric space and let
	$T:X\to \mathcal{C}(X)$. Suppose that, for some $p\in\mathbb N$, the mapping
	$T^p:X\to \mathcal{C}(X)$ satisfies all the hypotheses of Theorem~\ref{MainSFix}. If $\text{SFix}(T)\neq\emptyset$, then
	$$
	\text{SFix}(T)=\text{Fix}(T)=\{x^*\}.$$
\end{theorem}

\begin{proof}
Let $w\in \operatorname{SFix}(T)$. Then
$T^pw=\{w\}$,
so
$w\in \operatorname{SFix}(T^p)$.
Since $T^p$ satisfies Theorem~\ref{MainSFix},
$\operatorname{Fix}(T^p)=\{x^*\}$.
Thus
$w=x^*$.
Then for any $u\in \operatorname{Fix}(T)$,
$u\in \operatorname{Fix}(T^p)$,
hence
$u=x^*$.
Therefore
$\operatorname{SFix}(T)=\operatorname{Fix}(T)=\{x^*\}$.
\end{proof}

\begin{remark}
	\label{strictTp}
	The  existence of a strict fixed point of $T^p$ alone is insufficient to guarantee a strict fixed point of $T$, i.e., the assumption $\text{SFix}(T)\neq\emptyset$ in Theorem~\ref{Main6}
	cannot, in general, be replaced $\text{SFix}(T^p)\neq\emptyset$. Indeed, for a multivalued mapping, a strict fixed point of an iterate
	need not be a strict fixed point of the original mapping. For example,
	let $X=\{z,w\}$ and define
	$Tz=\{w\}, Tw=\{z\}.$
	Then $	T^2z=T(Tz)=Tw=\{z\},$
	so that $z\in\text{SFix}(T^2)$. However,
	$Tz=\{w\}\neq\{z\},$
	and hence $z\notin\text{SFix}(T)$. Therefore, $	\text{SFix}(T^2)\neq \emptyset$ does not imply $\text{SFix}(T)\neq \emptyset$. This demostrates the necessesity of the strict fixed point assumption on $T$ in Theorem \ref{Main6}.
\end{remark}

\begin{corollary}
	\label{cor*}
	Let $(X,d)$ be a complete $b$-metric space with constant $s \geq 1$ and let $T : X \to \mathcal{C}(X)$ be a multivalued mapping. Suppose that $\text{SFix}(T)\neq \emptyset$ and that, for some $p\in\mathbb{N}$,  there exist nonnegative constants $a, b, c, e, f \in [0,1)$ with $
	a + b + c + 2s\min\{e, f\} < 1 $ and \eqref{gammac}.  Suppose $T^p:X\to\mathcal{C}(X)$ satisfies inequality \eqref{A4} for all $x,y\in X$. If at least one of the following conditions holds:
\begin{enumerate}[label=(\alph*)]
	\item $T^p$ is $h$-upper semicontinuous;
	\item $s(c+e)<1$;
	\item $s(b+f)<1$,
\end{enumerate}
then
$
\text{SFix}(T)=\text{Fix}(T).$
Moreover, if $a+e+f<1,$
then
$\text{SFix}(T)=\text{Fix}(T)=\{x^*\}$.
\end{corollary}
\begin{proof}
	Define the nonnegative functions in Theorem \ref{Main2} by
	$$
	a(t) \equiv a, \quad b(t) \equiv b, \quad c(t) \equiv c, \quad e(t) \equiv e, \quad f(t) \equiv f, \quad t \geq 0.$$
	Then
	$$
	a(t) + b(t) + c(t) + 2s \min\{e(t), f(t)\} = a + b + c + 2s \min\{e, f\} < 1,
	$$
	so that the contractive condition \eqref{A4} is satisfied by $T^p$. Moreover, the function $\gamma$ associated with $T^p$ in Theorem~\ref{Main2} is constant and is given by \eqref{gammac}. 
	The assumptions in (a), (b), and (c) correspond, respectively, to the $h$-upper semicontinuity and the two limiting conditions required in Theorem~\ref{Main2}. Hence, $T^p$ satisfies all the hypotheses of Theorem~\ref{Main2}. Since $\text{SFix}(T)\neq\emptyset$, Theorem~\ref{Main6} yields
	$
	\text{SFix}(T) = \text{Fix}(T) = \{x^*\}.$
\end{proof}

\begin{corollary}
Let $(X,d)$ be a complete $b$-metric space with constant $s\geq 1$, and let $T:X\to X$ be a mapping. Suppose that there exist nonnegative functions
$a,b,c,e,f:[0,\infty)\to[0,1)$ such that $a(t)+b(t)+c(t)+2s\min\{e(t),f(t)\}<1,
~ t\geq 0.$ Assume further that \eqref{gamma} holds with $ \displaystyle \sup_{t\geq 0}\gamma(t)<1,$ and that, for some $p\in\mathbb{N}$, the iterate $T^p:X\to X$ satisfies \eqref{A3} for all $x,y\in X$. In addition, suppose that $a(t)+e(t)+f(t)<1,~ t\geq 0,$ and that at least one of the following conditions is satisfied:
	\begin{enumerate}[label=(\alph*)]
		\item $T$ is continuous; 
			\item $\displaystyle s	\limsup_{t\to0^+}\left(c(t)+e(t)\right)<1$;
		\item $\displaystyle s	\limsup_{t\to0^+}\left(b(t)+f(t)\right)<1$. 
	\end{enumerate}
Then $T$ has a unique fixed point in $X$. 
\end{corollary}
\begin{proof}
Regard $T$ as the multivalued mapping $x\mapsto\{T x\}$. Since $T$ is continuous in case (a), $T^p$ is continuous and hence $h$-upper semicontinuous. In cases (b) and (c), the corresponding conditions in Theorem \ref{Main2} are satisfied.  Therefore, 
 	$T^p$ yields $ \operatorname{Fix}(T^p)\neq\emptyset.$ Since $a(t)+e(t)+f(t)<1$ for all $t\geq0$,
 	the uniqueness part of Theorem  \ref{MainSFix} gives 
 	$\text{Fix}(T^p)=\{x^*\}$ for some $x^*\in X.$
 Since
 $T^p x^*=x^*$, we have
 $T^p(Tx^*)=T(T^p x^*)=Tx^*.$ Thus, $Tx^*\in \text{Fix}(T^p)=\{x^*\}$ and consequently $Tx^*=x^*.$	Hence $x^*\in \text{Fix}(T)$. Finally, let $w\in \text{Fix}(T)$. Then $Tw=w$, and therefore
 $T^p w=w.$ Thus, $w\in \text{Fix}(T^p)=\{x^*\},$ which implies $w=x^*$. Consequently, $\text{Fix}(T)=\{x^*\}.$
 Thus $T$ has a unique fixed point.
 \end{proof}

The following example illustrates the role of the second iterate in Theorem~\ref{Main6}. We construct a multivalued mapping $T$ that fails the Hardy--Rogers type contractive condition due to its behavior at $x=1$, whereas its second iterate $T^2$ is $h$-upper semicontinuous and satisfies the required contractive condition.

\begin{example}
	\label{last}
	Let $X=[0,1]$ and define $d(x,y)=|x-y|^3,~x,y\in X.$ Then $(X,d)$ is a complete $b$-metric space with coefficient $s=4$.	Define $T:X\to C(X)$ by
$$
T(x)=
\begin{cases}
	\left\{\dfrac{x}{2}\right\}, & x\in[0,1),\\[3mm]
	\left\{0,\dfrac12\right\}, & x=1.
\end{cases}
$$

Consider the coefficient functions
$$
q(t)=\frac{t}{100(1+t)},~t\geq0,
$$
and
$$ a(t)=q(t),~~ b(t)=c(t)=\frac14-q(t),
~~ e(t)=f(t)=\frac1{20}+\frac1{10}q(t).$$
Clearly, all five coefficient functions are nonconstant and take
values in $[0,1)$. Moreover,

\begin{align*}
	a(t)+b(t)+c(t)+2s\min\{e(t),f(t)\}=q(t)+2\left(\frac14-q(t)\right)
	+8\left(\frac1{20}+\frac1{10}q(t)\right)=\frac9{10}-\frac15q(t)<1.
\end{align*}

	We first show that $T$ does not satisfy the Hardy--Rogers type
contractive condition \eqref{A3}.
Consider $x=1$ and $y=1-\varepsilon$, where $0<\varepsilon<1/2$. Then
$$
T(1)=\left\{0,\frac12\right\},
~~
T(1-\varepsilon)
=
\left\{\frac{1-\varepsilon}{2}\right\}.$$
Hence
$$
H(T1,T(1-\varepsilon))=\max\left\{\left(\frac{1-\varepsilon}{2}\right)^3,
\left(\frac{1+\varepsilon}{2}\right)^3\right\}=\left(\frac{1+\varepsilon}{2}\right)^3.
$$
Moreover,
\begin{align*}
	d(1,1-\varepsilon)&=\varepsilon^3, ~D(1,T1)=\frac18,~~
	D(1-\varepsilon,T(1-\varepsilon))	=	\left(\frac{1-\varepsilon}{2}\right)^3,\\
	D(1,T(1-\varepsilon))&=\left(\frac{1+\varepsilon}{2}\right)^3,~~
	D(1-\varepsilon,T1)	=	\min\left\{(1-\varepsilon)^3,	\left(\frac12-\varepsilon\right)^3\right\}
	=	\left(\frac12-\varepsilon\right)^3.
\end{align*}
	Hence, the right-hand side of \eqref{A3} is
\begin{align*}
	R_\varepsilon
	&={}q(\varepsilon^3)\varepsilon^3
	+\left(\frac14-q(\varepsilon^3)\right)\frac18
	+\left(\frac14-q(\varepsilon^3)\right)
	\left(\frac{1-\varepsilon}{2}\right)^3
	+\left(\frac1{20}+\frac1{10}q(\varepsilon^3)\right)
	\left(\frac{1+\varepsilon}{2}\right)^3\\
	&+\left(\frac1{20}+\frac1{10}q(\varepsilon^3)\right)
	\left(\frac12-\varepsilon\right)^3.
\end{align*}
Since
$\displaystyle \lim_{\varepsilon\to0^+}q(\varepsilon^3)=0,$
we obtain
$$
\lim_{\varepsilon\to0^+}R_\varepsilon
=
\frac1{32}+\frac1{32}+\frac1{160}+\frac1{160}
=\frac3{40}< \frac{1}{8}=
\lim_{\varepsilon\to0^+}
H(T1,T(1-\varepsilon))
$$
Since $\frac3{40}<\frac18,$ there exists $\varepsilon_0>0$ such that $H(T1,T(1-\varepsilon))>R_\varepsilon$ for every $0<\varepsilon<\varepsilon_0$. Therefore, $T$ does not
satisfy \eqref{A3}.

We now consider the second iterate. For $x\in[0,1)$,
$$
T^2(x)=\left\{\frac{x}{4}\right\} \text{ while }T^2(1)
=
T\left(\left\{0,\frac12\right\}\right)
=
\left\{0,\frac14\right\}.
$$
Thus
$$
T^2(x)=
\begin{cases}
	\left\{\dfrac{x}{4}\right\}, & x\in[0,1),\\[3mm]
	\left\{0,\dfrac14\right\}, & x=1.
\end{cases}
$$

We claim that the mapping $T^2$ is $h$-upper semicontinuous. Indeed, for any $x \in [0,1)$ and $x_n \to x$, we have
$$
\delta(T^2 x_n, T^2 x) = \frac{1}{64}|x_n - x|^3 \to 0.$$
For $x = 1$ and $x_n \to 1$, if $x_n \in [0,1)$, then
$$
\delta(T^2 x_n, T^2(1))
= \min\left\{ \frac{x_n^3}{64}, \frac{(1-x_n)^3}{64} \right\} \to 0.$$
If $x_n = 1$, then $\delta(T^2 x_n, T^2(1)) = 0$. Hence, in all cases,
$$
\lim_{n \to \infty} \delta(T^2 x_n, T^2 x) = 0.$$ Therefore, $T^2$ is $h$-upper semicontinuous.

We next verify that $T^2$ satisfies the Hardy--Rogers type
contractive condition \eqref{A3}. Suppose first that
$x,y\in[0,1)$. If $x=y$, then
$	H(T^2x,T^2y)=0,$
and hence \eqref{A3} holds trivially. Now suppose that $x\ne y$.
Then $$H(T^2x,T^2y)=\frac1{64}|x-y|^3,~~ D(x,T^2x)=\left(\frac{3x}{4}\right)^3 =\frac{27}{64}x^3,~~
D(y,T^2y)=\frac{27}{64}y^3.
$$
Since $	b(t)=c(t)=\frac14-q(t)>\frac15$ and $	|x-y|^3\le x^3+y^3,	$
we obtain
\begin{align*}
	H(T^2x,T^2y)&=\frac1{64}|x-y|^3 \leq \frac1{64}(x^3+y^3)\\
	&<\frac15\frac{27}{64}(x^3+y^3)<b(d(x,y))D(x,T^2x)
	+c(d(x,y))D(y,T^2y).
\end{align*}
Thus, \eqref{A3} holds for all $x,y\in[0,1)$.

Next, let $x=1$ and $y\in[0,1)$. Since
$$
T^2(1)=\left\{0,\frac14\right\},
~~
T^2(y)=\left\{\frac y4\right\},
$$
we have
$$H(T^2(1),T^2(y))=\max\left\{\frac{y^3}{64},\frac{(1-y)^3}{64}\right\}\leq\frac1{64}.
$$
Moreover,
$$D(1,T^2(1))=d\left(1,\frac14\right)=\frac{27}{64}.$$
Since $b(t)>\frac15$, we obtain
$$
b(d(1,y))D(1,T^2(1))>\frac15\frac{27}{64}=\frac{27}{320}>\frac1{64}.$$
Consequently,

\begin{align*}
	H(T^2(1),T^2(y))<b(d(1,y))D(1,T^2(1))\leq b(d(1,y))D(1,T^2(1))
	+c(d(1,y))D(y,T^2(y)).
\end{align*}
Thus, \eqref{A3} holds for $x=1$ and $y\in[0,1)$. The case
$x\in[0,1)$ and $y=1$ follows by symmetry. The case $x=y=1$ is
trivial. Thus, $T^2$ satisfies the Hardy--Rogers type contractive
condition \eqref{A3}.

For the corresponding function $\gamma$ in Theorem \ref{Main2}, we have $b(t)=c(t)$ and $e(t)=f(t)$, and hence the two terms defining $\gamma$ coincide.
Thus,
\begin{align*}
	\gamma(t)
	&=
	\frac{a(t)+b(t)+4e(t)}
	{1-c(t)-4e(t)}=
	\frac{9+8q(t)}{11+12q(t)}.
\end{align*}
Since
$$
\frac{d}{dq}
\left(\frac{9+8q}{11+12q}\right)
=
-\frac{20}{(11+12q)^2}<0,
$$
and $q(t)\geq0$, we obtain
$$
\sup_{t\geq0}\gamma(t)=\gamma(0)=\frac9{11}<1.	$$
Morover, 
$a(t)+e(t)+f(t)	<1$ for every $t\geq 0$, and  $0\in SFix(T)$ because $T(0)=\{0\}$. Since
$T^2:X\to\mathcal{C}(X)$ is $h$-upper semicontinuous, all the hypotheses of Theorem~\ref{Main6} are satisfied with
$p=2$. Consequently,
$$SFix(T)=Fix(T)=\{0\}.$$
\end{example}

	\begin{remark}
			Example~\ref{last} illustrates that the strict fixed point conclusions obtained in Theorem~\ref{Main6} may apply in situations where the corresponding hypotheses of  \cite[Theorem 3.6]{Chaib2026} and
			\cite[Corollary 4.8]{Li} are not satisfied. In particular, for $s=4$,
		$$
		s\limsup_{t\to0^+}\bigl(c(t)+e(t)\bigr)
		=4\left(\frac14+\frac1{20}\right)
		=\frac65>1,
	\text{ and }
		s\limsup_{t\to0^+}\bigl(b(t)+f(t)\bigr)
		=\frac65>1.$$
		Nevertheless, $T^2$ is $h$-upper semicontinuous, so the $h$-upper semicontinuity alternative in Theorem 4.14 applies.
	\end{remark}

\begin{remark}
	The results obtained in this paper extend the framework of \cite[Theorem 3.6]{Chaib2026} in several directions.
	\begin{itemize}
		\item The constant coefficients in the Hardy--Rogers type contractive condition are replaced by nonnegative functions depending on the interpoint distance, leading to a distance-dependent contractive framework; see Theorem \ref{Main2}.
		
		\item The contraction factor is formulated through the piecewise function $\gamma(t)$, which corresponds directly to the two cases $e(t)\leq f(t)$ and $f(t)<e(t)$ arising in the proof of Theorem \ref{Main2}.
		
		\item The fixed point existence assumptions are formulated in terms of the asymptotic conditions
		$$
		s\limsup_{t\to0^+}\bigl(c(t)+e(t)\bigr)<1
		\text{ and }
		s\limsup_{t\to0^+}\bigl(b(t)+f(t)\bigr)<1,
		$$
		which, in the constant-coefficient setting, reduce to
		$$
		s(c+e)<1
		\text{ and }
		s(b+f)<1,
		$$
		respectively.
		
		\item The contractive condition may be imposed on a suitable finite iterate $T^p$, rather than necessarily on the original mapping $T$; see Theorem \ref{Main6}. The existence and uniqueness of the strict fixed point were established.
		
		\item The $h$-upper semicontinuity of the relevant mapping or iterate provides an alternative route to fixed point existence, without requiring either of the above asymptotic conditions.
	\end{itemize}
	
	These features provide a unified framework for studying distance-dependent Hardy--Rogers type multivalued mappings and their finite iterates in complete \(b\)-metric spaces.
\end{remark}

\section{Conclusion}

In this paper, we developed a distance-dependent Hardy–Rogers type framework for multivalued mappings in complete $b$-metric spaces. The principal result employs a piecewise contraction factor corresponding to the cases $e(t)\leq f(t)$ and $f(t)<e(t)$, with the positivity of the associated denominators made explicit. Under suitable alternative assumptions involving $h$-upper semicontinuity or asymptotic limsup conditions on the coefficient functions, we established fixed point existence.

We further investigated strict fixed points and obtained conditions under which the fixed point and strict fixed point sets coincide and reduce to a singleton. In addition, we established a fixed point transfer result for finite iterates. Specifically, when the original multivalued mapping has a strict fixed point and a suitable iterate satisfies the hypotheses of the corresponding fixed point theorem, the uniqueness of the fixed point of the iterate transfers to the original mapping.

The results also yield constant-coefficient and single-valued specializations. The examples illustrate the applicability of the framework and demonstrate, in particular, that the Hardy–Rogers type contractive condition may hold for a suitable iterate even when it fails for the original mapping. They also highlight the usefulness of $h$-upper semicontinuity as an alternative assumption for fixed point existence.

These results provide a framework for further investigation of distance-dependent Hardy–Rogers type contractions and their iterates in generalized metric structures.
 \begin{comment}
 
 %\subsection*{Acknowledgements} The authors are grateful to Professor Vladimir Andrievskii for his help and support during the writing of this paper.

\noindent
\textbf{Declarations}
 \vskip 0.2in

\noindent
\textbf{Ethical Approval}\\
Not applicable.\\

\noindent
\textbf{Availability of Data and Material}\\
 Not applicable.\\

\noindent
\textbf{Conflict of Interests}\\
 The authors declare that they have no conflict of interest.\\

\noindent \textbf{Funding}\\
The authors received no funding for this research.		\\

\noindent
\textbf{{Authors' Contributions}}\\
Conceptualization of the article was given by HSA, methodology
by HSA, formal analysis, investigation and writing–original draft preparation by HSA and OTM, 
 writing–review and editing by OTM, project administration and supervision  by OTM. 
 All authors have accepted responsibility for the entire content of this manuscript
and approved its submission.		
 \end{comment}


\begin{thebibliography}{99}
	
	\bibitem{Banach} S. Banach. Sur les opérations dans les ensembles abstraits et leur application aux équations intégrales, {\it Fund. Math.} {\bf 3} (1922), 133-181.
		
		\bibitem{Hardy} G. E. Hardy \& T. D. Rogers. A generalization of a fixed point theorem of Reich, {\it Canad. Math. Bull.} {\bf 16} (1973), no. 2, 201-206.
		
		\bibitem{Kannan} R. Kannan, Some results on fixed points, \textit{Bull. Calcutta Math. Soc.}\textbf{ 60} (1968), 71--76.
		
			\bibitem{Chatterjea} S. K. Chatterjea, Fixed-point theorems, \textit{Dokl. Bulg. Acad. Sci.} \textbf{25 }(1972), 727--730.
			
		\bibitem{Reich} S. Reich, Kannan's fixed point theorem, \textit{Bull. Univ. Mat. Italiana} \textbf{4} (1971), 1--11.
	
			\bibitem{Nadler} S. B. Nadler. Multivalued contraction mappings, {\it Pacific J. Math.} {\bf 30} (1969), no. 2, 475-488.
		
			\bibitem{Czerwik1998} S. Czerwik. Nonlinear set-valued contraction mappings in $b$-metric spaces, \textit{Atti Semin. Mat. Fis. Univ. Modena} \textbf{46} (1998), 263--276.
		

	\bibitem{Bakhtin} I. Bakhtin. The contraction mapping principle in quasimetric spaces, {\it Funct. Anal.} {\bf 30} (1989), 26-37.
	
	\bibitem{Czerwik} S. Czerwik. Contraction mappings in $b$-metric spaces, {\it Acta Math. Inform. Univ. Ostraviensis} {\bf 1} (1993), no. 1, 5-11.
	
	\bibitem{Lael} F. Lael, N. Saleem, \& M. Abbas. On the fixed points of multivalued mappings in $b$-metric spaces and their application to linear systems, {\it UPB Sci. Bull. Ser. A} {\bf 82} (2020), no. 4, 121-130.
		
	\bibitem{Khojasteh} F. Khojasteh and V. Rakočević. Some new common fixed point results for generalized contractive multi-valued non-self-mappings, \textit{Appl. Math. Lett.} \textbf{25} (2012), 287--293.
	
	\bibitem{Leyew} B. T. Leyew \& O. T. Mewomo. Common fixed point results for generalized orthogonal $F$-Suzuki contraction for a family of multivalued mappings in orthogonal $b$-metric spaces, {\it Commun. Korean Math. Soc.} {\bf 37} (2022), no. 4, 1147-1170.
	
	\bibitem{Joonaghany} G. H. Joonaghany, O. T. Mewomo, \& F. Khojasteh. Weak KKH-Caristi type expansion on complete $b$-metric and $b$-ordered metric spaces and its consequences, {\it Afr. Mat.} {\bf 32} (2021), no. 7-8, 1281-1294.
	
		\bibitem{Andres} J. Andres, J. Fišer, and L. Gorniewicz. Fixed points and sets of multivalued contractions: An advanced survey with some new results, \textit{Fixed Point Theory} \textbf{22} (2021), no. 1, 15.
	
	\bibitem{Aliouche} A. Aliouche \& T. Hamaizia. Common fixed point theorems for multivalued mappings in $b$-metric spaces with an application to integral inclusions, {\it J. Anal.} {\bf 30} (2022), no. 1, 43-62.
	
	\bibitem{Karapinar2022} E. Karapınar, A. Ali, A. Hussain, \& H. Aydi. On interpolative Hardy-Rogers type multivalued contractions via a simulation function, {\it Filomat} {\bf 36} (2022), no. 8, 2847-2856.
	
	\bibitem{AliHussain} A. Ali, A. Hussain, \& Z. D. Mitrović. Multivalued Hardy-Rogers type $Z\Theta$-contraction and generalized simulation functions, {\it Filomat} {\bf 36} (2022), no. 1, 1-14.
	
	\bibitem{AliKhan} B. Ali, A. A. Khan, \& A. Hussain. Best proximity points of multivalued Hardy-Roger's type (cyclic) contractive mappings of $b$-metric spaces, {\it J. Math.} {\bf 2022} (2022), 11.
	
	\bibitem{Errai} Y. Errai, E. M. Marhrani, \& M. Aamri. Some new results of interpolative Hardy–Rogers and Ćirić–Reich–Rus type contraction, {\it J. Math.} {\bf 2021} (2021).
	
	\bibitem{Patel} D. Patel, M. Younis, D. Singh, \& O. P. Chauhan. Banach contraction induced by Hardy-Rogers contractions: computation and applications, {\it Comput. Appl. Math.} {\bf 45} (2026), 358.
	
	\bibitem{Aleksic} S. Aleksic, Z. D. Mitrovic, \& S. Radenovic. On some recent fixed point results for single and multivalued mappings in b-metric spaces, {\it Fasc. Math.} {\bf 61} (2018), 5-16.
	
	\bibitem{Derouiche} D. Derouiche and H. Ramoul. New fixed point results for $F$-contractions of Hardy-Rogers type in $b$-metric spaces with applications, \textit{J. Fixed Point Theory Appl.} \textbf{22} (2020), no. 4, Paper No. 86.
	
	\bibitem{Chifu} C. Chifu and G. Petruşel. Fixed point results for multi-valued Feng-Liu contractions in $b$-metric spaces, \textit{Fixed Point Theory} \textbf{25} (2024), no. 2,  507--518.
		

		\bibitem{Al-Izeri} A. M. Al-Izeri and K. Latrach. A note on fixed point theory for multivalued mappings, \textit{Fixed Point Theory} \textbf{24} (2023), no. 1, 233--240.
		
			\bibitem{Roldan2018}
		A. F. Roldan Lopez de Hierro, F. Khojasteh, and S. Moradi.
		On quasi-contractive multivalued mappings' open problem in complete metric spaces,
		\textit{Math. Methods Appl. Sci.} \textbf{41} (2018), no. 17, 7147--7157.
	
	\bibitem{Chaib2026} R. Chaib, F. Merghadi, \& Z. Mouhoubi. Fixed point theorems for multivalued mappings of Hardy-Rogers type in $b$-metric spaces, {\it Fixed Point Theory} {\bf 27} (2026), no. 1, 209-228.
	
	\bibitem{Chifu2017} C. Chifu and G. Petruşel. Fixed point results for multivalued Hardy–Rogers contractions in $b$-metric spaces, \textit{Filomat} \textbf{31} (2017), no. 8, 2499--2507.
	
	\bibitem{Singh2005} S. L. Singh, C. Bhatnagar, and S. N. Mishra. Stability of iterative procedures for multivalued maps in metric spaces, \textit{Demonstr. Math.} \textbf{38} (2005), no. 4, 905--916.
	
	\bibitem{Covitz} H. Covitz and S. B. Nadler. Multi-valued contraction mappings in generalized metric spaces, \textit{Israel J. Math.} \textbf{8} (1970), no. 1, 5--11.
	
		\bibitem{Miculescu} R. Miculescu \& A. Mihail. New fixed point theorems for set-valued contractions in $b$-metric spaces, {\it J. Fixed Point Theory Appl.} {\bf 19} (2017), no. 3, 2153-2163.
		
		\bibitem{Chaib} R. Chaib, F. Merghadi, \& Z. Mouhoubi. Improvement of fixed point theorems for Hardy–Rogers contraction type in $b$-metric spaces without F-contraction assumption, {\it Rend. Circ. Mat. Palermo (2)} {\bf 72} (2023), no. 8, 4209-4237.
	
		\bibitem{Li} F. H. S. F. Li and S. M. Lu. Several new fixed point results for multi-valued quasi-contractions in $b$-metric spaces, \textit{Fixed Point Theory} \textbf{26} (2025), no. 1, 197.
		
			\bibitem{Mitrovic} Z. D. Mitrovic. Fixed point results in $b$-metric space, {\it Fixed Point Theory} {\bf 20} (2019), no. 2, 559-566.
\end{thebibliography}
	\end{document}